\documentclass[letterpaper, 10 pt, conference]{ieeeconf}
\IEEEoverridecommandlockouts

\usepackage{amssymb,amsmath}

\usepackage{amsthm}
\usepackage{mathtools}
\usepackage{multirow}
\usepackage{booktabs}
\usepackage{graphicx}
\usepackage{comment}
\usepackage[sort,compress,noadjust]{cite}
\usepackage[font=footnotesize]{subcaption}
\usepackage[font=footnotesize]{caption}

\usepackage{amsfonts}
 \usepackage{tabularx}
\usepackage{graphicx}
\usepackage{array}
\usepackage{float}
\usepackage{hhline}
\usepackage{algorithmic}
\usepackage[linesnumbered, ruled, vlined]{algorithm2e}
\usepackage{setspace}
\usepackage{colortbl}
\usepackage{arydshln}
\SetCommentSty{mycommfont}

\DeclareMathOperator*{\argmin}{arg\,min}

\usepackage{accents}
\newcommand{\ubar}[1]{\underaccent{\bar}{#1}}

\usepackage{placeins}

\theoremstyle{plain}
\newtheorem{theorem}{Theorem}
\theoremstyle{plain}
\newtheorem{proposition}{Proposition}
\theoremstyle{plain}

\theoremstyle{plain}
\newtheorem{corollary}{Corollary}
\theoremstyle{definition}

\newtheorem{remark}{Remark}
\newtheorem{example}{Example}

\newtheoremstyle{definition}{}{}{}{}{\bfseries}{.}{.5em}{\thmname{#1}\thmnumber{ #2}\thmnote{ (#3)}}
\theoremstyle{definition}
\newtheorem{definition}{Definition}

\newcommand{\real}{\mathbb{R}}

\newcommand{\map}[3]{#1:#2 \rightarrow #3}

\newcommand{\setdefb}[2]{\big\{#1 \; | \; #2\big\}}

\renewcommand{\bf}{\mathbf{f}}
\newcommand{\bg}{\mathbf{g}}

\newcommand{\bk}{\mathbf{k}}

\newcommand{\bn}{\mathbf{n}}

\newcommand{\bu}{\mathbf{u}}
\newcommand{\bv}{\mathbf{v}}
\newcommand{\bw}{\mathbf{w}}
\newcommand{\bx}{\mathbf{x}}
\newcommand{\by}{\mathbf{y}}

\newcommand{\bA}{\mathbf{A}}
\newcommand{\bB}{\mathbf{B}}

\newcommand{\bG}{\mathbf{G}}
\newcommand{\bH}{\mathbf{H}}
\newcommand{\bI}{\mathbf{I}}
\newcommand{\bJ}{\mathbf{J}}

\newcommand{\bL}{\mathbf{L}}

\newcommand{\bV}{\mathbf{V}}

\newcommand{\bZ}{\mathbf{Z}}

\newcommand{\bLambda}{\boldsymbol{\Lambda}}
\newcommand{\balpha}{\boldsymbol{\alpha}}

\newcommand{\bzero}{\mathbf{0}}

\newcommand{\bOmega}{\boldsymbol{\Omega}}

\newcommand{\Bc}{\mathcal{B}}
\newcommand{\Cc}{\mathcal{C}}

\newcommand{\Kc}{\mathcal{K}}

\newcommand{\Sc}{\mathcal{S}}

\DeclareMathOperator{\sym}{sym}

\DeclareMathOperator{\range}{range}
\DeclareMathOperator{\tr}{tr}

\usepackage[colorlinks,urlcolor=blue]{hyperref}
\usepackage[capitalize]{cleveref}
\crefformat{equation}{(#2#1#3)}
\Crefformat{equation}{Equation~(#2#1#3)}
\Crefname{equation}{Equation}{Eqs.}

\title{\LARGE \textbf{Optimal Decay Matrix Control Barrier Functions:\\Soundness and Completeness for Matrix Safety Constraints }}

\author{ Samuel G. Gessow$^{1}$, Albert Wu$^{2}$, Aaron D. Ames$^2$, and Brett T. Lopez$^{1}$
\thanks{$^1$Authors are with VECTR Laboratory, University of California, Los Angeles, Los Angeles, CA, USA, {\tt\footnotesize \{sgessow, btlopez\}@ucla.edu}}
\thanks{$^2$Authors are with AMBER Lab, California Institute of Technology, Pasadena, CA 91125, USA, {\tt\footnotesize \{albertwu, ames\}@caltech.edu}}
}

\begin{document}
\maketitle
\thispagestyle{empty}
\begin{abstract}
Many safety requirements, such as keeping a manipulator away from kinematic singularities, a network connected, or a system observable, amount to keeping a state-dependent matrix positive definite.
Matrix control barrier functions (MCBFs) enforce such requirements with a semidefinite program that minimally modifies a nominal control input.
The MCBF constraint bounds how fast the safety margin, the smallest eigenvalue of the matrix, may decay.
Yet a fixed decay rule is conservative, rejecting safe controls that drive the state quickly toward, but never across, the boundary of the safe set.
Scalar optimal decay (SOD) addresses this issue by scaling the decay term by a scalar weight that is optimized together with the control.
For scalar barrier functions, SOD then rejects a control at the boundary only if the control would decrease the margin.
For MCBFs, however, SOD can reject controls that leave the margin unchanged, and near the safety boundary SOD accepts such controls only at a high cost.
We introduce matrix optimal decay (MOD), which replaces the scalar weight with a matrix multiplier while keeping the semidefinite program convex.
MOD is provably sound and complete, certifying at the boundary exactly the controls that do not decrease the margin.
If some control can strictly increase the margin at every boundary point, the MOD-MCBF filtered controller is continuous and renders the safe set forward invariant.
We also add an optional recovery condition that drives the system toward the safe set after a violation without affecting soundness or completeness.
On a planar two-link arm avoiding singularities, MOD-MCBF certifies safe motions that both a fixed decay rule and SOD reject, and tracks an end-effector target more closely than SOD.
\end{abstract}

\section{Introduction}\label{sec:introduction}
\suppressfloats[t]
\begin{figure}[t]
    \centering
    \begin{subfigure}[t]{0.49\columnwidth}
        \centering
        \includegraphics[width=\linewidth]{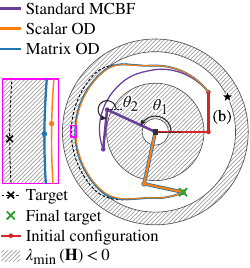}
        \caption{End-effector paths}
        \label{fig:hero_closed_loop}
    \end{subfigure}\hfill
    \begin{subfigure}[t]{0.49\columnwidth}
        \centering
        \includegraphics[width=\linewidth]{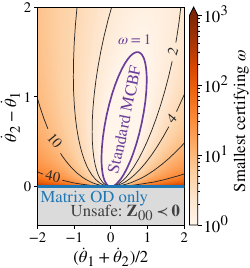}
        \caption{Certified joint velocities at the outer boundary.}
        \label{fig:hero_rates}
    \end{subfigure}
    \caption{Benefits of matrix optimal decay (MOD) on a planar two-link robot that must stay away from kinematic singularities while tracking an end-effector position target.
    (a) The standard MCBF, with a fixed decay function, rejects safe motions that exceed certain joint velocities, while scalar optimal decay (SOD) rejects some safe motions and degrades tracking.
    (b)~Joint velocities (rad/s) at the outer boundary configuration marked by the star in (a). The grey region violates the tangent cone condition~\eqref{eq:tangentcone}. The safe region is colored by the smallest scalar decay weight~\eqref{eq:arm_omega_boundary} that certifies each velocity, and the standard MCBF certifies only the velocities inside the purple curve. Rotating both links together about joint 1 (blue line) preserves the constraint and is certified only by MOD; toward this line, the scalar decay weight required at any nonzero rotation rate grows without bound.
}
    \label{fig:hero}
\end{figure}

Matrix control barrier functions (MCBFs)~\cite{ong2026matrix} provide an effective framework for enforcing safety constraints that arise as positive definiteness of state-dependent matrices, such as information matrices~\cite{gessow2024information}, graph Laplacians~\cite{ong2023nonsmooth,ong2026matrix}, and Jacobian singular values~\cite{forghani2026task}. 
In a typical control barrier function (CBF) safety filtering setup, an optimization problem is solved to find the control input closest to a nominal input, subject to the CBF safety constraint. 
Scalar CBFs (SCBFs) impose a linear constraint that requires a state-dependent safety margin to decrease no faster than a decay function evaluated at the current margin.
MCBFs impose the analogous limit on a state-dependent matrix, whose safety margin is its smallest eigenvalue, through a single matrix semidefinite constraint.
Enforcing the matrix inequality directly avoids differentiating eigenvalues, which can be nonsmooth where they coincide~\cite{glotfelter2017nonsmooth,gessow2024information}.

A critical decision in synthesizing CBF constraints is choosing the decay function, which determines how quickly the safety margin may decrease inside the safe set.
A poor choice can make the constraint needlessly conservative, or even infeasible, at points inside the safe set~\cite{ong2025properties}.
Scalar optimal decay (SOD)~\cite{zeng2021safety,gurriet2020scalable,ong2025properties} avoids committing to a single fixed choice of decay function by scaling it with a nonnegative decay weight that is optimized together with the control input.
Wherever the safety margin is positive, a large enough weight admits any control.
For SCBFs, the margin is zero on the boundary, so for every weight the constraint reduces to requiring that the margin not decrease, which is necessary for safety~\cite{nagumo1942lage,blanchini2008set}.
SOD therefore rejects a control only if it violates this condition at the boundary.

For MCBFs, different eigenvalues can change at different rates, while a fixed decay function constrains the rate of each one.
This can make the safety filter infeasible even when a safety-preserving control exists.
Optimal decay relaxes the bounds on eigenvalues away from zero, which can restore feasibility and guarantee safety~\cite{gessow2026high}.
Yet SOD can still reject controls that preserve every eigenvalue at the boundary.
For example, a control that rotates the matrix without changing its eigenvalues leaves the safety margin unchanged, yet SOD can reject it.
SOD is also conservative inside the safe set: unlike for SCBFs, even controls that leave the safety margin unchanged in MCBFs require a larger scalar decay weight to certify as the state approaches the boundary, resulting in sacrificing nominal control tracking performance.

In this work, we present matrix optimal decay (MOD) for MCBFs, a generalization of SOD to the matrix case. 
We prove that MOD-MCBF is \emph{sound}, certifying at the boundary only controls that do not decrease the safety margin, and \emph{complete}, certifying every such control.
A recovery condition can also be imposed that drives the system toward the safe set if safety is violated, without affecting the boundary guarantees.
We demonstrate MOD-MCBF through singularity avoidance for a planar two-link arm (\cref{fig:hero}), showing that MOD certifies safe motions that are rejected by existing formulations and provides better tracking performance.

\section{Preliminaries}
\subsection{Notation}
The set of real numbers is denoted $\real$, with $\real^n$ denoting the $n$-dimensional Euclidean space, and $\real_{\geq0}$ the set of nonnegative scalars.
The set of real symmetric $p\times p$ matrices is denoted $\mathbb{S}^p$, and the cone of positive definite symmetric matrices is denoted $\mathbb{S}^p_+$. For $\bA,\bB\in\mathbb{S}^p$ we write $\bA\succeq\bB$ ($\bA\succ\bB$) when $\bA-\bB$ is positive semidefinite (definite).
We say a locally Lipschitz function $\alpha:\real\to\real$ is an extended class-$\mathcal{K}$ function $(\alpha \in \mathcal{K}^e)$ if it is strictly monotonically increasing and satisfies $\alpha(0)=0$.
Given a square matrix $\bA \in \real^{n \times n}$, the set of its eigenvalues is denoted $\lambda(\bA)$, the smallest and largest eigenvalues are $\lambda_{\mathrm{min}}(\bA)$ and $\lambda_{\mathrm{max}}(\bA)$, and the kernel and range of $\bA$ are $\ker\bA$ and $\range\bA$.
The symmetric part of a square matrix is $\sym(\bA)\triangleq\tfrac{1}{2}(\bA+\bA^\top)$, and $\|\bA\|_F$ denotes the Frobenius norm.
The Lie derivative of a scalar function $\bx \mapsto h(\bx)$ along a vector field $\bx \mapsto \bf(\bx)$ is denoted $L_\bf h(\bx)= \nabla h( \bx)^\top \bf(\bx)$. 
Similarly, the Lie derivative of a matrix-valued function $\bx \mapsto \bH(\bx)\in\mathbb{S}^p$ along a vector field $\bx \mapsto \bf(\bx)$ is denoted $\bL_\bf \bH(\bx)$, and it is equivalent to a matrix of element-wise Lie derivatives, i.e., $[\bL_\bf \bH(\bx)]_{ij}=L_\bf H_{ij}(\bx)$, $i,j\in\{1,\dots,p\}$.

\subsection{Matrix Control Barrier Functions}
CBFs provide a framework for enforcing state constraints in nonlinear systems through feedback control. 
Consider the control-affine system:
\begin{equation}\label{sys:ctrl_affine}
    \dot \bx = \bf(\bx) + \bg(\bx) \bu,
\end{equation}
where $\bx \in \real^n$ is the state and $\bu\in\real^m$ is the control input. 
The drift vector field $\map{\bf}{\real^n}{\real^n}$ and the control matrix $\map{\bg}{\real^n}{\real^{n\times m}}$ are assumed to be smooth.  

A safety specification is encoded through a \textit{safety constraint set} $\Sc \subseteq \real^n$ where the control objective is to ensure that system trajectories $t\mapsto\bx(t)$ remain in $\Sc$ for all time $t\geq 0$, i.e., to render $\Sc$ \textit{forward invariant} under the closed-loop dynamics.
Achieving this requires $\Sc$ to be \textit{control invariant}: for every initial condition $\bx_0\in\Sc$, there exists a control input $t\mapsto \bu(t)$ such that the ensuing trajectory is contained within $\Sc$.
Since safety constraints are typically prescribed by the task and may not naturally be control invariant, it is common to instead construct a control-invariant subset $\Cc\subseteq \Sc$. CBFs provide a tool for certifying control invariance of such sets.

\begin{definition}
Let $\Sc \subseteq \real^n$ be a \textit{safety constraint set}. A set $\Cc$ is said to be \emph{safe} for \cref{sys:ctrl_affine} if $\Cc\subseteq \Sc$ and $\Cc$ is control invariant.
\end{definition}

MCBFs encode the set $\Cc$ through a symmetric matrix-valued function $\map{\bH}{\real^n}{\mathbb{S}^p}$:
\begin{equation}
\Cc = \setdefb{\bx\in\real^n}{\bH(\bx)\succeq \bzero}.
\label{eq:safe_set}
\end{equation}
Throughout, $\Bc_H\triangleq\setdefb{\bx\in\Cc}{\lambda_{\min}(\bH(\bx))=0}$ denotes the \textit{matrix-constraint boundary}.
The boundary of the safe set satisfies $\partial\Cc\subseteq\Bc_H$, although $\Bc_H$ may also contain interior points of $\Cc$.
At points $\bx\in\Cc\setminus\Bc_H$ we have $\bH(\bx)\succ\bzero$, so $\balpha(\bH(\bx))\succ\bzero$, where $\balpha$ is the spectral application of $\alpha$, i.e., $\balpha(\bH)\triangleq\sum_i\alpha(\lambda_i)\bv_i\bv_i^\top$ for $\bH=\sum_i\lambda_i\bv_i\bv_i^\top$~\cite{ong2026matrix}.
Aligned with this safe set definition, MCBFs are developed using the ordering induced by the positive semidefinite cone.
\begin{definition}[cf.~\cite{ong2026matrix}]
Let $\Cc \subset \mathbb{R}^n$ be defined by a continuously differentiable function $\map{\bH}{\real^n}{\mathbb{S}^p}$ as in~\eqref{eq:safe_set}. The function $\bH$ is a \emph{matrix control barrier function} (MCBF) for \cref{sys:ctrl_affine} if there exists $\alpha\in\Kc^e$ such that, for each $\bx\in\Cc$, there exists a $\bu\in \mathbb{R}^m$ satisfying
\begin{equation}
    \dot{\bH}(\bx,\bu) \triangleq\bL_\bf\bH(\bx)+\sum_{j=1}^m \bL_{\bg_j}\bH(\bx) u_j \succ -\balpha(\bH(\bx)),
\end{equation}
where $\bg_j$ is a column of $\bg$.
\label{def:MCBF}
\end{definition}
The matrix $\balpha(\bH)$ shares the eigenvectors of $\bH$ and, when $\bH\succeq\bzero$, satisfies $\ker\balpha(\bH)=\ker\bH$.
This fact is central to the boundary analysis of \cref{sec:scalarnotcomplete,sec:boundary}.
For $p=1$, $\bH=h$ and \cref{def:MCBF} reduces to the scalar CBF definition~\cite{ames2017control}.
MCBFs ensure the possibility of designing a continuous controller $\map{\bk}{\real^n}{\real^m}$ that satisfies $\dot\bH(\bx,\bk(\bx))\succeq-\balpha(\bH(\bx))$, which guarantees that $\bH$ remains positive semidefinite along closed-loop trajectories~\cite{ong2026matrix}.
In particular, a matrix CBF-based safety filter selects the safe control input closest to a continuous nominal controller $\map{\bk_{\rm nom}}{\real^n}{\real^m}$ through the semidefinite program (SDP):
\begin{equation}\label{eq:cbf_sdp}
\begin{aligned}
\bk(\bx) = \argmin_{\bu\in\real^m} \quad & \|\bu - \bk_{\rm nom}(\bx)\|^2 \\
\text{s.t.} \quad & \dot\bH(\bx,\bu)\succeq-\balpha(\bH(\bx)).
\end{aligned}
\end{equation}
The admissible control set at $\bx$ consists of all control inputs satisfying the constraint in~\eqref{eq:cbf_sdp}.
The strict barrier condition in \cref{def:MCBF} ensures continuity of the optimal control input. For $p=1$, this filter reduces to the scalar CBF quadratic program (CBF-QP).
A useful specific version of the MCBF is the exponential MCBF where $\alpha$ is chosen to be linear, i.e., $\alpha(r) = \gamma r$ with $\gamma>0$, in which case $\balpha(\bH) = \gamma \bH$ and no spectral decomposition of $\bH$ is required.

\subsection{Optimal Decay}
\label{sec:optimal_decay}
The choice of $\alpha$ in \cref{def:MCBF} determines the controls admitted by the barrier condition and therefore affects the deviation of the safety filter from $\bk_{\rm nom}(\bx)$.
Finding a fixed $\alpha$ that makes this condition feasible throughout the safe set can be difficult, even in the scalar case.
At states where the input cannot affect $\dot h$, the required decay can be arbitrarily large for the same value of $h$, exceeding any fixed $\alpha(h)$~\cite{ong2025properties}.
To address this difficulty, the optimal decay modification~\cite{zeng2021safety,gurriet2020scalable,ong2025properties} introduces a scalar decision variable $\omega\geq0$ that scales $\alpha$, allowing the effective decay to be optimized jointly with the control input.
For the scalar case $p=1$, this gives the following definition.
\begin{definition}
    \label{def:od-cbf}
    Let $\Cc \subset \mathbb{R}^n$ be the $0$-superlevel set of a continuously differentiable function 
    $h : \mathbb{R}^n \rightarrow \mathbb{R}$. 
    The function $h$ is an {\emph{optimal decay control barrier function}} (OD-CBF) for \cref{sys:ctrl_affine} if there exists $\alpha\in\Kc^e$ such that, for each $\bx \in \Cc$, there exist a $\bu \in \real^m$ and $\omega \geq 0$ such that
    \begin{equation}
    \label{eq:od_CBF_cond}
        L_\bf h(\bx) +  L_{\bg} h(\bx) \bu > - \omega \alpha(h(\bx)).
    \end{equation}
\end{definition}
This can be incorporated into the scalar specialization of~\eqref{eq:cbf_sdp} by scaling its decay term by $\omega$ and augmenting the cost with a penalty $\rho_\omega(\omega-\omega_d)^2$ on deviation from a nominal rate $\omega_d>0$.
If $h(\bx)>0$, then $\alpha(h(\bx))>0$, so every $\bu\in\real^m$ satisfies~\eqref{eq:od_CBF_cond} for sufficiently large $\omega$.
If $h(\bx)=0$, the decay term vanishes, and the non-strict version of~\eqref{eq:od_CBF_cond} reduces to
\begin{equation}\label{eq:nagumo_scalar}
    L_\bf h(\bx)+L_\bg h(\bx)\bu\geq0,
\end{equation}
which is Nagumo's condition~\cite{nagumo1942lage}.
Since the interior condition can always be met by increasing $\omega$, certifying $h$ as an OD-CBF reduces to finding a strictly inward control at each boundary point~\cite[Lem.~3]{ong2025properties}.

\section{Scalar Optimal Decay: Sound but Incomplete}\label{sec:scalarnotcomplete}
For $p>1$, the difficulty of finding a feasible fixed $\balpha$ is compounded by restrictions absent in the scalar CBF case.
At a boundary point, $\bH(\bx)$ may have both zero and positive eigenvalues.
Safety is only governed by the directions associated with the zero eigenvalues, whereas the positive eigenvalues may evolve more freely without immediately threatening safety. 
However, because the function $\balpha$ is applied to the entire matrix,
it limits the evolution of the strictly positive eigenvalues and may exclude control directions that can safely drive the zero eigenvalues inward.
Consequently, the MCBF condition can be infeasible even when a strictly inward control exists at every boundary point~\cite{gessow2026high}. 
The SOD mechanism in \cref{def:od-cbf} extends directly to matrix constraints by scaling the spectral class-$\Kc$ term with a single scalar decision variable.
\begin{definition}[cf.~\cite{gessow2026high}]
Let $\Cc \subset \mathbb{R}^n$ be defined by a continuously differentiable function $\map{\bH}{\real^n}{\mathbb{S}^p}$ as in~\eqref{eq:safe_set}. The function $\bH$ is a \emph{scalar optimal decay matrix control barrier function} (SOD-MCBF) for \cref{sys:ctrl_affine} if there exists $\alpha\in\Kc^e$ such that, for each $\bx\in\Cc$, there exist a $\bu\in \mathbb{R}^m$ and an $\omega \geq 0$ satisfying
\begin{equation}
    \bL_\bf\bH(\bx)+\sum_{j=1}^m \bL_{\bg_j}\bH(\bx) u_j \succ -\omega \balpha(\bH(\bx)).
\end{equation}
\label{def:Scalar_od_MCBF}
\end{definition}
The corresponding SOD safety filter optimizes the control and decay rate jointly:
\begin{equation}\label{eq:scalar_od_sdp}
\begin{aligned}
    (\bk(\bx),\omega^\star(\bx))
    ={}&\argmin_{\bu\in\real^m,\,\omega\geq0}
    \|\bu-\bk_{\rm nom}(\bx)\|^2\\
    &\qquad+\rho_\omega(\omega-\omega_d)^2\\
    \text{s.t.}\quad
    &\dot\bH(\bx,\bu)+\omega\balpha(\bH(\bx))\succeq\bzero,
\end{aligned}
\end{equation}
where $\rho_\omega>0$ weights deviation from the nominal decay rate $\omega_d>0$.
Since $\balpha(\bH)$ is positive definite on the positive eigenspace of $\bH$, varying $\omega$ provides additional slack in these eigendirections. 
Thus, the restriction imposed by the positive eigendirections can be relaxed so that they no longer preclude control directions needed to satisfy the safety condition along the kernel of $\bH$. 
Feasibility is therefore governed by the existence of an appropriate safety-preserving control direction at the boundary. 
However, it remains unclear whether a single scalar decay variable $\omega$ is sufficiently expressive to admit every such control.

Independently of how the decay term is parameterized, Nagumo's theorem motivates the safety requirement: the dynamics must point into the safe set or along its boundary~\cite{nagumo1942lage,blanchini2008set}.
At a boundary point $\bx\in\Bc_H$, Nagumo's condition requires the state velocity $\dot\bx=\bf(\bx)+\bg(\bx)\bu$ to lie in the tangent cone of the safe set.
For the matrix constraint $\bH(\bx)\succeq\bzero$, every such control necessarily satisfies
\begin{equation}\label{eq:tangentcone}
    \bV_0^\top\dot\bH(\bx,\bu)\bV_0\succeq\bzero,
\end{equation}
where the columns of $\bV_0$ form an orthonormal basis of $\ker\bH(\bx)$.
We call~\eqref{eq:tangentcone} the \emph{matrix tangent cone condition} and analyze it at all points in $\Bc_H$.
Forward invariance of the state-space set $\Cc$ is established separately in \cref{sec:invariance}.

SOD can exclude controls satisfying~\eqref{eq:tangentcone}, imposing an additional restriction beyond the matrix tangent-cone condition.

\begin{example}\label{ex:scalar_counterexample}
Consider a state $\bx$ and a control $\bu$ such that
\begin{equation}\label{eq:scalar_counterexample}
    \bH(\bx)=\begin{bmatrix}0&0\\0&\lambda\end{bmatrix},\qquad
    \dot\bH(\bx,\bu)=\begin{bmatrix}0&b\\b&c\end{bmatrix},\quad
    \lambda>0,\;b\neq0.
\end{equation}
Here, $\bV_0=[1,0]^\top$ and $\bV_0^\top\dot\bH\bV_0=0$, so~\eqref{eq:tangentcone} holds for this control choice.
However, $\det(\dot\bH+\omega\balpha(\bH))=-b^2<0$ for every $\omega\geq0$, so no scalar decay rate certifies this control.
\end{example}

To make this precise, a condition is called \emph{sound} if it certifies only controls satisfying~\eqref{eq:tangentcone}, and \emph{complete} if it certifies every such control.
SOD is sound since $\balpha(\bH)\bV_0=\bzero$, so restricting its matrix inequality to the kernel gives~\eqref{eq:tangentcone}.
It is not complete, as shown by~\eqref{eq:scalar_counterexample}.

To characterize this limitation, fix $\bx\in\Bc_H$ and $\bu\in\real^m$.
If $\bH(\bx)=\bzero$, SOD reduces to $\dot\bH(\bx,\bu)\succeq\bzero$ and is complete, so the remaining case is a boundary point where $\bH(\bx)$ has a nonempty positive eigenspace.
Let $\bV = [\bV_0 \ \bV_1]$ be an orthogonal eigenbasis of $\bH(\bx)$.
Define the transformed matrices $\tilde \bH \triangleq \bV^\top\bH\bV$ and $\widetilde{\dot{\bH}} \triangleq \bV^\top\dot\bH\bV$.
Then
\begin{equation}\label{eq:blocksplit}
    \tilde\bH=\begin{bmatrix}\bzero&\bzero\\\bzero&\bLambda_1\end{bmatrix}\!,\;\;
    \balpha(\tilde\bH)=\begin{bmatrix}\bzero&\bzero\\\bzero&\bA\end{bmatrix}\!,\;\;
    \widetilde{\dot{\bH}}=\begin{bmatrix}\bZ_{00}&\bZ_{01}\\\bZ_{01}^\top&\bZ_{11}\end{bmatrix}\!,
\end{equation}
with $\bLambda_1\succ\bzero$ and $\bA=\balpha(\bLambda_1)\succ\bzero$. The blocks $\bZ_{ij}$ depend affinely on $\bu$.
In this basis,~\eqref{eq:tangentcone} is $\bZ_{00}\succeq\bzero$, while scalar decay only changes the positive-eigenspace block $\bZ_{11}$.

The following proposition identifies the additional restriction on the cross block $\bZ_{01}$.

\begin{proposition}\label{prop:scalargap}
Let $\bx\in\Bc_H$ and $\bu\in\real^m$, with blocks as in~\eqref{eq:blocksplit}. There exists $\omega\geq0$ with $\dot\bH+\omega\balpha(\bH)\succeq\bzero$ if and only if
\begin{equation}\label{eq:albert}
    \textup{(i)}\;\bZ_{00}\succeq\bzero,\qquad
    \textup{(ii)}\;(\bI-\bZ_{00}\bZ_{00}^{+})\bZ_{01}=\bzero,
\end{equation}
where $\bZ_{00}^{+}$ is the Moore--Penrose pseudoinverse.
\end{proposition}
\begin{proof}
The constraint $\dot\bH+\omega\balpha(\bH)\succeq\bzero$ is equivalent to
$\left[\begin{smallmatrix}\bZ_{00}&\bZ_{01}\\\bZ_{01}^\top&\bZ_{11}+\omega\bA\end{smallmatrix}\right]\succeq\bzero$.
By the generalized Schur complement conditions for positive semidefiniteness with a singular diagonal block~\cite[App.~A.5.5]{boyd2004convex}, \cite{zhang2006schur}, this holds if and only if (i), (ii), and
\textup{(iii)}~$\bZ_{11}+\omega\bA-\bZ_{01}^\top\bZ_{00}^{+}\bZ_{01}\succeq\bzero$.
Conditions (i) and (ii) do not depend on $\omega$, while (iii) holds for all sufficiently large $\omega\geq0$ since $\bA\succ\bzero$.
\end{proof}

Condition (ii) requires the columns of $\bZ_{01}$ to lie in $\range\bZ_{00}$.
When $\bZ_{00}\succ\bzero$, we have $\bZ_{00}^{+}=\bZ_{00}^{-1}$, so $\bI-\bZ_{00}\bZ_{00}^{+}=\bzero$ and condition (ii) holds for any $\bZ_{01}$.
SOD therefore certifies every control satisfying $\bZ_{00}\succ\bzero$.
The gap appears when the kernel block is itself singular and the columns of $\bZ_{01}$ are not in its range. For example, a grazing control with $\bZ_{00}=\bzero$ and $\bZ_{01}\neq\bzero$ violates condition (ii).

Even when a scalar certificate exists, the required decay can be arbitrarily large.
The following proposition quantifies this with a lower bound.
\begin{proposition}\label{prop:scalarcost}
Let $\bx\in\Cc$ and $\bu\in\real^m$, and write the matrices in
the basis of~\eqref{eq:blocksplit}. When $\bH(\bx)\succ\bzero$
the kernel-associated blocks are empty and
$\bZ_{01}^\top\bZ_{00}^{-1}\bZ_{01}\triangleq\bzero$, and when
$\bH(\bx)$ is singular suppose that $\bZ_{00}\succ\bzero$.
For $\omega\geq0$, the condition $\dot\bH+\omega\balpha(\bH)\succeq\bzero$ holds if and only if
\begin{equation}\label{eq:scalarcost}
    \omega\geq
    \max\Big\{0,\;
    \lambda_{\max}\big(
    \bA^{-\frac12}
    \big(
    \bZ_{01}^\top\bZ_{00}^{-1}\bZ_{01}
    -\bZ_{11}
    \big)
    \bA^{-\frac12}
    \big)\Big\}.
\end{equation}
In particular, if $\lambda_{\min}(\bZ_{00})\to0$ along a sequence
of points at which $\bH$ is singular, while
$\|\bZ_{01}^\top\bv\|\geq\delta>0$ for unit eigenvectors $\bv$ of
$\bZ_{00}$ associated with $\lambda_{\min}(\bZ_{00})$ and $\bA$,
$\bZ_{11}$ remain bounded, then
$\dot\bH+\omega\balpha(\bH)\succeq\bzero$ requires
\begin{equation}\label{eq:scalarcost_blowup}
    \omega
    \geq
    \frac{1}{\lambda_{\max}(\bA)}
    \left(
    \frac{\|\bZ_{01}^\top\bv\|^2}
    {\lambda_{\min}(\bZ_{00})}
    -\lambda_{\max}(\bZ_{11})
    \right)
\end{equation}
where the bound goes to infinity as $\lambda_{\min}(\bZ_{00})\to0$.
\end{proposition}

\begin{proof}
For $\bH\succ\bzero$, \eqref{eq:scalarcost} follows by congruence with $\bA^{-1/2}$ under the stated convention.
Otherwise, $\bZ_{00}\succ\bzero$, and the Schur complement~\cite{zhang2006schur} gives
\[
    \dot\bH+\omega\balpha(\bH)\succeq\bzero
    \quad\Longleftrightarrow\quad
    \omega\bA
    \succeq
    \bZ_{01}^\top\bZ_{00}^{-1}\bZ_{01}
    -\bZ_{11}.
\]
Congruence with $\bA^{-1/2}$ then gives~\eqref{eq:scalarcost}.

For~\eqref{eq:scalarcost_blowup}, set $\lambda=\lambda_{\min}(\bZ_{00})$ and $\bw=\bZ_{01}^\top\bv$.
Since $\bZ_{00}^{-1}\succeq\bv\bv^\top/\lambda$, evaluating the preceding inequality at $\bw$ gives
\begin{align*}
    \omega\,\bw^\top\bA\bw
    &\geq
    \bw^\top\bZ_{01}^\top\bZ_{00}^{-1}\bZ_{01}\bw
    -\bw^\top\bZ_{11}\bw\\
    &\geq \frac{\|\bw\|^4}{\lambda}
    -\lambda_{\max}(\bZ_{11})\|\bw\|^2.
\end{align*}
Using $\omega\geq0$ and $\bw^\top\bA\bw\leq\lambda_{\max}(\bA)\|\bw\|^2$, and dividing by $\|\bw\|^2>0$, yields~\eqref{eq:scalarcost_blowup}.
Divergence follows from $\|\bw\|\geq\delta$, $\lambda\to0$, and boundedness of $\bA$ and $\bZ_{11}$.
\end{proof}

Thus, scalar decay can make a control increasingly costly to certify even before certification becomes infeasible.
When decay is penalized, this can cause the filter to modify an otherwise safe control.

\Cref{sec:example} demonstrates these limitations in a two-link arm.
The following section introduces a matrix-valued decay variable that remains sound and recovers completeness.

\section{Matrix Optimal Decay: Sound and Complete}\label{sec:matrixod}
To retain all controls that preserve safety at the boundary, we generalize the scalar decay variable to a matrix-valued variable $\bOmega\in\real^{p\times p}$.
\begin{definition}\label{def:MODMCBF}
Let $\Cc \subset \mathbb{R}^n$ be defined by a continuously differentiable function $\map{\bH}{\real^n}{\mathbb{S}^p}$ as in~\eqref{eq:safe_set}. 
The function $\bH$ is a \emph{matrix optimal decay MCBF} (MOD-MCBF) for \cref{sys:ctrl_affine} if there exists $\alpha\in\Kc^e$ such that, for each $\bx\in\Cc$, there exist a $\bu\in \mathbb{R}^m$ and an $\bOmega \in \real^{p\times p}$ satisfying
\begin{equation}\label{eq:odcond_strict}
    \dot\bH(\bx,\bu)+\bOmega^\top\balpha(\bH(\bx))+\balpha(\bH(\bx))\,\bOmega\;\succ\;\bzero.
\end{equation}
\end{definition}
A strict inequality in \cref{def:MODMCBF} ensures regularity of the safety filter.
The admissible controls at the boundary are characterized by the non-strict version of~\eqref{eq:odcond_strict}
\begin{equation}\label{eq:odcond}
    \dot\bH(\bx,\bu)+\bOmega^\top\balpha(\bH(\bx))+\balpha(\bH(\bx))\,\bOmega\;\succeq\;\bzero.
\end{equation}

Soundness and completeness are understood in the sense of \cref{sec:scalarnotcomplete}, and \cref{sec:boundary} shows that the matrix multiplier recovers the completeness lost by SOD.
Since $\balpha(\bH)$ is symmetric, the decay term equals $2\sym\big(\balpha(\bH)\bOmega\big)$.
Setting $\bOmega=\tfrac12\bI$ recovers the standard MCBF condition, while $\bOmega=\tfrac\omega2\bI$, $\omega\geq0$, recovers SOD.
Hence, every control certified by either formulation is also certified by~\eqref{eq:odcond} at the same state.

\subsection{Soundness and Completeness}\label{sec:boundary}
Fix $\bx\in\Bc_H$ and a control $\bu$, and use the eigenbasis and blocks introduced in~\eqref{eq:blocksplit}.
If $\bH(\bx)=\bzero$, both~\eqref{eq:odcond} and~\eqref{eq:tangentcone} reduce to $\dot\bH(\bx,\bu)\succeq\bzero$, so the remaining case is a boundary point where $\bH(\bx)$ has a nonempty positive eigenspace.
Define $\tilde\bOmega\triangleq\bV^\top\bOmega\bV$ with conformable blocks $\bOmega_{ij}$. In this basis, the decay term of~\eqref{eq:odcond} has the block form
\begin{equation}\label{eq:decayblocks}
    \tilde\bOmega^\top\balpha(\tilde\bH)+\balpha(\tilde\bH)\tilde\bOmega=
    \begin{bmatrix}\bzero & \bOmega_{10}^\top\bA\\ \bA\bOmega_{10} & \bA\bOmega_{11}+\bOmega_{11}^\top\bA\end{bmatrix}.
\end{equation}
Two important observations follow from~\eqref{eq:decayblocks}.
First, the kernel-kernel block of the decay term is zero for every $\bOmega$.
Second, the blocks $\bOmega_{00}$ and $\bOmega_{01}$ do not appear in the constraint since they only multiply the zero block of $\balpha(\tilde\bH)$.

\begin{theorem}\label{thm:equiv}
Let $\bx\in\Bc_H$ and $\bu\in\real^m$. Then
\begin{multline*}
    \exists\,\bOmega\in\real^{p\times p}:\;
    \dot\bH+\bOmega^\top\balpha(\bH)+\balpha(\bH)\bOmega\succeq\bzero\\
    \iff\;
    \bV_0^\top\dot\bH\bV_0\succeq\bzero.
\end{multline*}
\end{theorem}
\begin{proof}
($\Rightarrow$) Let $\bv$ be such that $\bv^\top\bH\bv=0$. Since $\bH\succeq\bzero$, $\bv^\top\bH\bv=\|\bH^{1/2}\bv\|^2=0$ forces $\bH\bv=\bzero$, and since $\balpha$ acts spectrally with $\alpha(0)=0$, also $\balpha(\bH)\bv=\bzero$.
Both terms of the decay contribution involve $\balpha(\bH)\bv$:
\[
    \bv^\top\big(\bOmega^\top\balpha(\bH)+\balpha(\bH)\bOmega\big)\bv
    =2(\bOmega\bv)^\top\balpha(\bH)\bv=0,
\]
regardless of $\bOmega$.
Hence, if $\bG\triangleq\dot\bH+\bOmega^\top\balpha(\bH)+\balpha(\bH)\bOmega\succeq\bzero$ for some $\bOmega$, then $0\leq\bv^\top\bG\bv=\bv^\top\dot\bH\bv$ for every $\bv\in\ker\bH=\range\bV_0$, which is~\eqref{eq:tangentcone}.

($\Leftarrow$) Suppose $\bZ_{00}=\bV_0^\top\dot\bH\bV_0\succeq\bzero$, and choose
\[
    \bOmega_{10}=-\bA^{-1}\bZ_{01}^\top,\qquad \bOmega_{11}=c\,\bA^{-1},\qquad c>0,
\]
with $\bOmega_{00},\bOmega_{01}$ arbitrary.
Using $\bA=\bA^\top$, we have $\bA\bOmega_{10}=-\bZ_{01}^\top$, $\bOmega_{10}^\top\bA=-\bZ_{01}$, and $\bA\bOmega_{11}+\bOmega_{11}^\top\bA=2c\bI$, so by~\eqref{eq:blocksplit} and~\eqref{eq:decayblocks}
\[
    \widetilde{\dot{\bH}}+\tilde\bOmega^\top\balpha(\tilde\bH)+\balpha(\tilde\bH)\tilde\bOmega=
    \begin{bmatrix}\bZ_{00}&\bzero\\\bzero&\bZ_{11}+2c\bI\end{bmatrix}.
\]
For sufficiently large $c$, $\bZ_{11}+2c\bI\succ\bzero$, so the block matrix is positive semidefinite. 
Setting $\bOmega=\bV\tilde\bOmega\bV^\top$, the result follows by orthogonal congruence from $\bZ_{00}\succeq\bzero$.
\end{proof}
\Cref{thm:equiv} can alternatively be established from the non-strict projection lemma~\cite[Thm.~1]{meijer2024non}.
The constructive proof shows how the matrix multiplier cancels the coupling that scalar decay cannot remove; \cref{tab:summary} compares the three decay families.

\begin{table}[t]
\vspace*{4pt}
\centering
\caption{Decay families at a singular boundary point.}
\label{tab:summary}
\footnotesize
\begin{tabular}{lll}
\toprule
Condition & Decay term & Feasibility\\
\hline
MCBF & $\balpha(\bH)$ & $\dot\bH+\balpha(\bH)\succeq\bzero$ \\[1pt]
SOD-MCBF & $\omega\,\balpha(\bH)$ & $\bZ_{00}\succeq\bzero$ \emph{and} $\bZ_{01}\in \range(\bZ_{00})$ \\[1pt]
MOD-MCBF & $2\sym(\balpha(\bH)\bOmega)$ & $\bZ_{00}\succeq\bzero$ (sound and complete) \\
\hline
\end{tabular}
\end{table}

\begin{remark}
\label{sec:lammin}
The scalarization $h_\lambda(\bx)=\lambda_{\min}(\bH(\bx))$ defines the same safe set~\eqref{eq:safe_set} but is generally nonsmooth.
At boundary points where the zero eigenvalue is simple, however, $h_\lambda$ is continuously differentiable in a neighborhood by classical perturbation theory~\cite{lax2007linear}. For a unit vector $\bv_0$ spanning $\ker\bH(\bx)$, its derivative is
\begin{equation}\label{eq:eigderiv}
    \dot h_\lambda(\bx,\bu)=\bv_0^\top\dot\bH(\bx,\bu)\bv_0.
\end{equation}
Since $h_\lambda(\bx)=0$ on the boundary, both the fixed-rate and OD-CBF conditions reduce there to $\dot h_\lambda(\bx,\bu)\geq0$, yielding the following equivalence.
\end{remark}

\begin{corollary}\label{thm:lammin}
Let $\bx\in\Bc_H$ with $\lambda_{\min}(\bH(\bx))=0$ simple, and fix $\bu\in\real^m$. The $\lambda_{\min}$ CBF condition holds at $(\bx,\bu)$ if and only if the MOD condition~\eqref{eq:odcond} is satisfiable at $(\bx,\bu)$ by some $\bOmega\in\real^{p\times p}$.
\end{corollary}
\begin{proof}
Both conditions are equivalent to $\bv_0^\top\dot\bH(\bx,\bu)\bv_0\geq0$: the former by~\eqref{eq:eigderiv}, the latter by \cref{thm:equiv} with $\bV_0=\bv_0$.
\end{proof}

\section{Forward Invariance}\label{sec:invariance}
A key observation is that the decay term influences the eigenvalues of $\bH$ strictly through the symmetric part of the multiplier: for a unit eigenvector $\bv$ of $\bH$ with eigenvalue $\lambda$, we have $\balpha(\bH)\bv=\alpha(\lambda)\bv$, so evaluating the quadratic form of~\eqref{eq:odcond} at $\bv$ gives
\begin{equation}\label{eq:perdirection}
    \bv^\top\dot\bH\bv \;\geq\; -2\,\alpha(\lambda)\,\bv^\top\sym(\bOmega)\,\bv.
\end{equation}
Specifically, each eigendirection receives its own decay rate in place of the single global rate of the scalar relaxation, while the skew part of $\bOmega$, which performs the cancellation in \cref{thm:equiv}, does not enter the eigenvalue evolution.
Forward invariance then follows in the same manner as for MCBFs and high-order MCBFs~\cite{ong2026matrix,gessow2026high}.

\begin{theorem}
\label{thm:invariance}
Let $\alpha\in\Kc^e$, and let $\map{\bk}{\real^n}{\real^m}$ and $\map{\bOmega}{\real^n}{\real^{p\times p}}$ be continuous functions satisfying
\begin{equation}\label{eq:odcond_closedloop}
    \dot\bH(\bx,\bk(\bx))+\bOmega(\bx)^\top\balpha(\bH(\bx))+\balpha(\bH(\bx))\,\bOmega(\bx)\succeq\bzero
\end{equation}
for all $\bx$ in an open neighborhood of $\Cc$.
Then $\Cc$ is forward invariant for the closed-loop system $\dot\bx=\bf(\bx)+\bg(\bx)\bk(\bx)$.
\end{theorem}
\begin{proof}
Let $t\mapsto\bx(t)$ be a closed-loop trajectory with $\bx(0)\in\Cc$, and consider any compact interval on which $\bx(t)$ remains in the neighborhood where~\eqref{eq:odcond_closedloop} holds.
By continuity of $\bOmega$, there exists $\bar\omega>0$ such that $\|\sym(\bOmega(\bx(t)))\|_2\leq\bar\omega$ on this interval. Thus,~\eqref{eq:perdirection} implies
\[
    \bv^\top\dot\bH\bv\;\geq\;-2\bar\omega\,\big|\alpha\big(\lambda_{\min}(\bH)\big)\big|
\]
for every unit minimum-eigenvalue eigenvector $\bv$. Local Lipschitzness of $\alpha$ and the standard comparison argument then imply forward invariance as in~\cite{ong2026matrix,gessow2026high}.
\end{proof}

In practice, the certified control input can be computed with the matrix optimal decay safety filter (MOD-SDP), which augments~\eqref{eq:cbf_sdp} with the decay variable and penalizes its deviation from a nominal value $\bOmega_d\triangleq\tfrac{\omega_d}{2}\bI$ with weight $\rho_\Omega>0$:
\begin{equation}\label{eq:odsdp}
\begin{aligned}
    \big(\bk(\bx),\bOmega(\bx)\big)&=\argmin_{\bu,\,\bOmega}\; \|\bu-\bk_{\rm nom}\|^2\\
    &\qquad+\rho_\Omega\|\bOmega-\bOmega_d\|_F^2\\
    \textup{s.t.}\; & \dot\bH(\bx,\bu)+\bOmega^\top\balpha(\bH)+\balpha(\bH)\,\bOmega\succeq\bzero,
\end{aligned}
\end{equation}
where $\omega_d>0$ is the nominal decay rate of the SOD framework ($\omega_d=1$ recovers the standard MCBF decay at $\bOmega=\bOmega_d$).
At fixed $\bx$, the constraint is a linear matrix inequality (LMI) jointly affine in $(\bu,\bOmega)$, so~\eqref{eq:odsdp} is a convex SDP. 
Note that for $p>1$ the SOD filter is also an SDP with the same cone dimension, while the matrix formulation has $p^2$ decay variables instead of one.

The existence of a continuous minimizer that enforces safety rests on a mild regularity property of the safe set.

\begin{definition}[cf.~\cite{gessow2026high}]\label{def:wellposed}
The safe set $\Cc$ in~\eqref{eq:safe_set} is \emph{well-posed} for~\eqref{sys:ctrl_affine} if for each $\bx\in\Bc_H$ there exists $\bu\in\real^m$ such that $\bV_0(\bx)^\top\dot\bH(\bx,\bu)\,\bV_0(\bx)\succ\bzero$, where the columns of $\bV_0(\bx)$ form an orthonormal basis of $\ker\bH(\bx)$.
\end{definition}

\begin{theorem}\label{thm:wellposed}
If $\Cc$ is well-posed for~\eqref{sys:ctrl_affine}, then $\bH$ is a MOD-MCBF, the MOD-SDP~\eqref{eq:odsdp} is strictly feasible on an open neighborhood of $\Cc$ with a unique and continuous minimizer $(\bk(\bx),\bOmega(\bx))$, and $\bu=\bk(\bx)$ renders $\Cc$ forward invariant.
\end{theorem}
\begin{proof}
At points in $\Bc_H$, the input from \cref{def:wellposed} with the multiplier constructed in the proof of \cref{thm:equiv} satisfies~\eqref{eq:odcond_strict} since now $\bZ_{00}\succ\bzero$. At points in $\Cc\setminus\Bc_H$, $\balpha(\bH)\succ\bzero$, so $\bOmega=\tfrac{\omega}{2}\bI$ suffices for $\omega$ large enough, and $\bH$ is a MOD-MCBF.
The remainder of the proof follows as in~\cite[Thm.~4]{gessow2026high}.
Strict feasibility extends to an open neighborhood of $\Cc$ by continuity, the MOD-SDP admits a unique and continuous minimizer there, and \cref{thm:invariance} applies.
\end{proof}

The input supplied by well-posedness has $\bZ_{00}\succ\bzero$, so \cref{prop:scalargap} guarantees that it also admits a scalar certificate.
Thus, both formulations are feasible and can enforce forward invariance under well-posedness.
Nevertheless, the scalar decay required to certify candidate controls can diverge as they approach the boundary, as characterized in \cref{prop:scalarcost}.

\section{Filter Invasiveness}\label{sec:invasiveness}
To compare filter invasiveness, fix $\bx$ and take $\omega_d=1$ so that $\bOmega_d=\tfrac12\bI$ reproduces the standard decay, and let
\begin{equation}\label{eq:stdfilter}
    \bu_{\rm std}=\argmin_{\bu\in\real^m}\;\|\bu-\bk_{\rm nom}\|^2
    \quad\textup{s.t.}\quad \dot\bH(\bx,\bu)+\balpha(\bH)\succeq\bzero,
\end{equation}
denote the standard MCBF filter, whose minimizer is unique when feasible since the cost is strictly convex over a closed convex set.
When the standard filter is feasible and $\omega_d=1$, the MOD-SDP deviates no further from the nominal control at the same state.

\begin{theorem}\label{thm:nonexpansive}
If $\omega_d=1$ and the standard filter~\eqref{eq:stdfilter} is feasible at $\bx$, then the MOD-SDP~\eqref{eq:odsdp} is feasible, attains its minimum at a unique $(\bu_{\rm MOD},\bOmega^\star)$, and
\begin{multline}\label{eq:nonexpansive}
    \|\bu_{\rm MOD}-\bk_{\rm nom}\|^2
    \leq \|\bu_{\rm MOD}-\bk_{\rm nom}\|^2+\rho_\Omega\|\bOmega^\star-\bOmega_d\|_F^2\\
    \leq \|\bu_{\rm std}-\bk_{\rm nom}\|^2 .
\end{multline}
\end{theorem}
\begin{proof}
The pair $(\bu_{\rm std},\bOmega_d)$ is feasible for~\eqref{eq:odsdp} since at $\bOmega=\bOmega_d$ the constraint is that of~\eqref{eq:stdfilter}, and its penalty term vanishes, so it attains the objective value $\|\bu_{\rm std}-\bk_{\rm nom}\|^2$.
The MOD optimum is no larger, which gives the right inequality, and dropping the nonnegative penalty gives the left.
The minimum is attained since $\rho_\Omega>0$ makes the cost jointly coercive in $(\bu,\bOmega)$ and the feasible set is closed, and it is unique since the cost is jointly strictly convex.
\end{proof}

The argument only uses the fact that $\bOmega_d$ is admissible and recovers the unrelaxed constraint, so it holds for any penalty weight $\rho_\Omega>0$.
Applying the same argument with $\bOmega$ restricted to $\tfrac{\omega}{2}\bI$ and the scalar penalty $\rho_\omega(\omega-\omega_d)^2$, with $\rho_\omega=\tfrac{p}{4}\rho_\Omega$ so that $\rho_\Omega\|\tfrac{\omega}{2}\bI-\bOmega_d\|_F^2=\rho_\omega(\omega-\omega_d)^2$ and the two penalties agree on the restriction, yields the ordering of the optimal objective values $J^\star$ of the MOD~\eqref{eq:odsdp}, SOD~\eqref{eq:scalar_od_sdp}, and standard~\eqref{eq:stdfilter} filters
\begin{equation}\label{eq:hierarchy}
    J^\star_{\textup{MOD}}\;\leq\;J^\star_{\textup{SOD}}\;\leq\;J^\star_{\textup{std}}\triangleq\|\bu_{\rm std}-\bk_{\rm nom}\|^2 .
\end{equation}
Note that~\eqref{eq:hierarchy} orders the total costs, not the control deviations, so it does not imply a pointwise ordering of the deviations between the scalar and matrix filters.

\section{Recovery Constraint}\label{sec:recovery}
If the system is initialized outside the safe set or is pushed outside by disturbances, the filter should actively drive it back, and the optimal decay relaxation by itself provides no such guarantee.
The SOD framework addresses this by imposing $\omega\geq \omega_d$~\cite{ong2025properties}, which requires at least the nominal decay rate.
The matrix analogue of this requirement is the additional constraint
\begin{equation}\label{eq:recovery}
    \sym(\bOmega)\succeq\beta\bI
\end{equation}
for a chosen recovery margin $\beta>0$, where $\beta=\tfrac12$ corresponds to $\omega\geq1$.
Constraint~\eqref{eq:recovery} is itself an LMI in $\bOmega$, so adding it to~\eqref{eq:odsdp} preserves convexity.
If an eigenvalue $\lambda$ of $\bH$ is negative, then $\alpha(\lambda)<0$, and whenever the augmented filter is feasible the per-direction bound~\eqref{eq:perdirection} yields
\begin{equation}
    \dot\lambda \geq -2\alpha(\lambda)\bv^\top \sym(\bOmega)\bv
\geq -2\beta\alpha(\lambda) > 0,
\end{equation}
so the decay term drives negative eigenvalues toward zero.
The derivative inequality holds wherever the eigenvalue branch is differentiable. At a repeated negative eigenvalue $\lambda$ with orthonormal eigenspace basis $\bV_\lambda$, the same conclusion holds for every eigenvalue directional derivative since
\[
    \bV_\lambda^\top\dot\bH\bV_\lambda\succeq-2\alpha(\lambda)\,\bV_\lambda^\top\sym(\bOmega)\,\bV_\lambda
    \succeq-2\beta\alpha(\lambda)\,\bI\succ\bzero .
\]
\begin{proposition}\label{prop:recovery}
Let $\beta>0$. Then the following hold:
\begin{enumerate}
    \item[(a)] for each $\bx\in\Bc_H$ and $\bu\in\real^m$, there exists $\bOmega\in\real^{p\times p}$ satisfying~\eqref{eq:odcond} and~\eqref{eq:recovery} $\Leftrightarrow$ $\bV_0^\top\dot\bH(\bx,\bu)\bV_0\succeq\bzero$;
    \item[(b)] if additionally $\beta\leq\omega_d/2$, then the conclusions of \cref{thm:nonexpansive} also hold when~\eqref{eq:recovery} is included in~\eqref{eq:odsdp}.
\end{enumerate}
\end{proposition}
\begin{proof}
(a) ($\Rightarrow$) Adding~\eqref{eq:recovery} only shrinks the admissible set of $\bOmega$, and soundness in \cref{thm:equiv} assumed nothing about $\bOmega$.
($\Leftarrow$) Take the choice $\bOmega_{10}=-\bA^{-1}\bZ_{01}^\top$ and $\bOmega_{11}=c\bA^{-1}$ from the proof of \cref{thm:equiv}, and use the free blocks to satisfy the bound by setting $\bOmega_{01}=-\bOmega_{10}^\top$ and $\bOmega_{00}=\mu\bI$.
The cross block of $\sym(\tilde\bOmega)$ then vanishes, leaving $\sym(\tilde\bOmega)=\mathrm{diag}(\mu\bI,\,c\bA^{-1})$.
Taking $\mu\geq\beta$ and $c\geq\max\{\beta\lambda_{\max}(\bA),\,c_0\}$, with $c_0$ large enough that $\bZ_{11}+2c_0\bI\succ\bzero$, and setting $\bOmega=\bV\tilde\bOmega\bV^\top$ satisfies~\eqref{eq:recovery} and the safety LMI simultaneously.
(b) The proof of \cref{thm:nonexpansive} used only feasibility of $(\bu_{\rm std},\bOmega_d)$, which is retained if and only if $\sym(\bOmega_d)=\tfrac{\omega_d}{2}\bI\succeq\beta\bI$, i.e., $\beta\leq\omega_d/2$. The nominal design $\omega_d=1$ and $\beta=\tfrac12$ satisfies this with equality.
\end{proof}

\section{Singularity Avoidance}\label{sec:example}
\subsection{System Setup}\label{sec:example_setup}
Consider a planar two-link arm with link lengths $l_1,l_2>0$ and state $\bx=(\theta_1,\theta_2)$, the absolute link angles measured counterclockwise from the horizontal as in \cref{fig:hero_closed_loop}, under velocity control $\dot\bx=\bu$.
The end effector is at $\by(\bx)=l_1[c_1,\,s_1]^\top+l_2[c_2,\,s_2]^\top$, where $s_i=\sin\theta_i$ and $c_i=\cos\theta_i$, and its Jacobian is
\begin{equation}\label{eq:arm_jacobian}
    \bJ(\bx)=\begin{bmatrix}-l_1s_1&-l_2s_2\\ l_1c_1&l_2c_2\end{bmatrix}=\begin{bmatrix}l_1\bn_1&l_2\bn_2\end{bmatrix},
\end{equation}
with $\bn_i=[-s_i,\,c_i]^\top$.
Since $\det\bJ=l_1l_2\sin(\theta_2-\theta_1)$, where $\theta_2-\theta_1$ is the angle at joint~2, the arm is singular when it is straight or folded.
To keep it away from these configurations, we impose a lower bound $0<\bar\sigma<\min(l_1,l_2)$ on the smallest singular value of $\bJ$ through the barrier
\begin{equation}\label{eq:singularitybarrier}
    \bH(\bx)\triangleq\bJ(\bx)\bJ(\bx)^\top-\bar\sigma^2\bI.
\end{equation}
Since the eigenvalues of $\bJ\bJ^\top$ are the squared singular values of $\bJ$~\cite[Sec.~3.9]{siciliano2009robotics}, $\lambda_{\min}(\bH(\bx))=\sigma_{\min}^2(\bJ(\bx))-\bar\sigma^2$, and hence $\bH(\bx)\succeq\bzero\Leftrightarrow\sigma_{\min}(\bJ(\bx))\geq\bar\sigma$.
We use the exponential decay $\alpha(\lambda)\triangleq\lambda$, so that $\balpha(\bH)=\bH$.

From~\eqref{eq:arm_jacobian}, $\bH$ has the constant trace $\tr\bH=l_1^2+l_2^2-2\bar\sigma^2>0$ and the determinant
\begin{equation}\label{eq:arm_det}
    \det\bH=(l_1^2-\bar\sigma^2)(l_2^2-\bar\sigma^2)-l_1^2l_2^2\cos^2(\theta_2-\theta_1),
\end{equation}
which depends on the configuration only through the angle $\theta_2-\theta_1$ at joint~2.
A symmetric $2\times2$ matrix is positive semidefinite if and only if its trace and determinant are nonnegative, so $\bx\in\Cc$ if and only if $\det\bH(\bx)\geq0$.
Substituting $\|\by\|^2=l_1^2+l_2^2+2l_1l_2\cos(\theta_2-\theta_1)$, a configuration is safe if and only if its end effector lies in the annulus
\begin{equation}\label{eq:arm_annulus}
\begin{aligned}
    &\ubar r\leq\|\by(\bx)\|\leq\bar r,\\
    &\bar r^2,\,\ubar r^2=l_1^2+l_2^2\pm2\sqrt{(l_1^2-\bar\sigma^2)(l_2^2-\bar\sigma^2)}.
\end{aligned}
\end{equation}
The end effector crosses the outer boundary $\bar r$ as the arm straightens and the inner boundary $\ubar r$ as it folds, into the unsafe configurations hatched in \cref{fig:hero_closed_loop}.
No safe motion can cross the straight or folded configurations, so $\theta_2-\theta_1$ is confined to one of the intervals between them; in all our runs, $0<\theta_2-\theta_1<\pi$.
Along trajectories, from~\eqref{eq:arm_jacobian} and~\eqref{eq:arm_det}
\begin{equation}\label{eq:arm_rates}
\begin{aligned}
    \tfrac{d}{dt}\det\bH&=l_1^2l_2^2\sin2(\theta_2-\theta_1)\,(\dot\theta_2-\dot\theta_1),\\
    \tfrac12\|\dot\bH\|_F^2&=l_1^4\dot\theta_1^2+l_2^4\dot\theta_2^2+2l_1^2l_2^2\dot\theta_1\dot\theta_2\cos2(\theta_2-\theta_1).
\end{aligned}
\end{equation}

At $\bx\in\Bc_H$, $\det\bH=0$ while $\tr\bH>0$, so the zero eigenvalue of $\bH$ is simple and $\lambda_{\max}(\bH)=\tr\bH$.
Differentiating $\det\bH=\lambda_{\min}(\bH)\,\lambda_{\max}(\bH)$ and applying~\eqref{eq:eigderiv} to the unit vector $\bv_0$ spanning $\ker\bH$ gives $\tfrac{d}{dt}\det\bH=\tr\bH\,\bv_0^\top\dot\bH\bv_0$.
The tangent cone condition~\eqref{eq:tangentcone} thus reduces to $\tfrac{d}{dt}\det\bH\geq0$, which is $(\dot\theta_2-\dot\theta_1)\sin2(\theta_2-\theta_1)\geq0$ by~\eqref{eq:arm_rates}.
It restricts only the joint 2 rate $\dot\theta_2-\dot\theta_1$: the arm may not straighten further at $\bar r$ ($\dot\theta_2\geq\dot\theta_1$) nor fold further at $\ubar r$ ($\dot\theta_2\leq\dot\theta_1$), while moving joint 1 with joint 2 fixed ($\dot\theta_1=\dot\theta_2$) is unrestricted.
Since $0<\bar\sigma<\min(l_1,l_2)$, setting $\det\bH=0$ in~\eqref{eq:arm_det} gives $\cos^2(\theta_2-\theta_1)\in(0,1)$, so $\sin2(\theta_2-\theta_1)\neq0$ and a suitable joint 2 rate makes $\tfrac{d}{dt}\det\bH>0$.
Such velocities are strictly inward: $\Cc$ is well-posed and \cref{thm:wellposed} applies.

Rotating about joint 1 ($\dot\theta_1=\dot\theta_2\neq0$) turns the unit vectors $\bn_1,\bn_2$ in~\eqref{eq:arm_jacobian} rigidly, so the eigenvalues of $\bH$ remain constant, yet $\dot\bH\neq\bzero$ wherever they are distinct.
On $\Bc_H$, this safe motion reproduces \cref{ex:scalar_counterexample}. Here, $\tilde\bH=\mathrm{diag}(0,\tr\bH)$; the diagonal entries of $\widetilde{\dot{\bH}}$ are zero, and the off-diagonal entries are $b$ with $|b|=|\dot\theta_1|\tr\bH\neq0$.

We choose a nominal controller that tracks an effector target $\by_d(t)$:
\begin{equation}\label{eq:arm_nominal}
    \bk_{\rm nom}(t,\bx)=\bJ(\bx)^{-1}\big(\dot\by_d(t)+\kappa\,(\by_d(t)-\by(\bx))\big),
\end{equation}
with gain $\kappa>0$, so that the unfiltered tracking error decays as $e^{-\kappa t}$.
The OD filters use $\omega_d=1$, so that $\bOmega_d=\tfrac12\bI$, the matched penalties $\rho_\omega=0.005$ and $\rho_\Omega=0.01$ of \cref{sec:invasiveness}, and the recovery constraints $\omega\geq\omega_d$ and~\eqref{eq:recovery} with $\beta=\tfrac12$, respectively.
Each command is held for $\Delta t=0.01$\,s; since $\dot\bx=\bu$, the next state is $\bx+\Delta t\,\bu$.
Each SDP is solved by Clarabel~0.11.1~\cite{goulart2024clarabel} through CVXPY~1.9.3~\cite{diamond2016cvxpy}, with gap and feasibility tolerances $10^{-9}$.

\subsection{Analysis of Expected Behavior}\label{sec:example_analysis}
We now examine how each filter handles the quantities of \cref{sec:example_setup}, in particular the joint 1 rotation, which preserves $\lambda_{\min}(\bH)$.

\emph{Matrix OD.}
The matrix multiplier
\begin{equation}\label{eq:rotationcertificate}
    \bOmega=\begin{bmatrix}1/2&-\dot\theta_1\\\dot\theta_1&1/2\end{bmatrix},
\end{equation}
whose skew part matches the rotation rate, cancels $\dot\bH$ exactly and reduces~\eqref{eq:odcond} to $\bH\succeq\bzero$; matrix OD therefore certifies the rotation at every $\bx\in\Cc$, on the boundary as in the interior.
Its symmetric part $\tfrac12\bI$ satisfies the recovery constraint~\eqref{eq:recovery} with $\beta=\tfrac12$, and $\|\bOmega-\tfrac12\bI\|_F=\sqrt2|\dot\theta_1|$ remains bounded at any rotation rate, independently of the distance to $\Bc_H$.

\emph{Scalar OD.}
Since $\tr\bH$ is constant, the trace of $\dot\bH+\omega\bH$ equals $\omega\tr\bH\geq0$ for every $\omega\geq0$, so the scalar OD condition $\dot\bH+\omega\bH\succeq\bzero$ reduces to $\det(\dot\bH+\omega\bH)\geq0$, with
\begin{equation}\label{eq:arm_scalar}
    \det(\dot\bH+\omega\bH)=\omega^2\det\bH+\omega\tfrac{d}{dt}\det\bH-\tfrac12\|\dot\bH\|_F^2.
\end{equation}
Setting $\det\bH=0$ in~\eqref{eq:arm_scalar}, per~\eqref{eq:scalarcost}, scalar OD certifies a strictly inward velocity at the boundary if and only if
\begin{equation}\label{eq:arm_omega_boundary}
    \omega\geq\frac{\|\dot\bH\|_F^2}{2\,\tfrac{d}{dt}\det\bH}.
\end{equation}
\cref{fig:hero_rates} illustrates this via its color code. The required decay~\eqref{eq:arm_omega_boundary} grows unboundedly as the joint 2 rate $\dot\theta_2-\dot\theta_1$ vanishes.
The joint 1 rotation reaches this limit: it gives $\tfrac{d}{dt}\det\bH=0$, so on $\Bc_H$ the determinant~\eqref{eq:arm_scalar} equals $-\tfrac12\|\dot\bH\|_F^2<0$ for every $\omega$.
Hence no scalar decay certifies the rotation at the boundary, as stated in \cref{ex:scalar_counterexample}.
At interior points, the rotation still gives $\tfrac{d}{dt}\det\bH=0$, and combining~\eqref{eq:arm_scalar} with~\eqref{eq:arm_rates} shows that the scalar condition holds if and only if
\begin{equation}\label{eq:rotation_scalar}
    \omega^2\det\bH\geq\dot\theta_1^2\big((l_1^2-l_2^2)^2+4l_1^2l_2^2\cos^2(\theta_2-\theta_1)\big).
\end{equation}
Scalar OD therefore certifies any rotation rate in the interior, but the decay it requires grows as $|\dot\theta_1|/\sqrt{\det\bH}$ and diverges as the state approaches $\Bc_H$.

\emph{Standard MCBF.}
The standard MCBF is scalar OD with $\omega=1$.
At the boundary, the rotation rate it certifies vanishes.
In the interior, setting $\omega=1$ in~\eqref{eq:rotation_scalar}, since the left side decreases and the right side increases with $\cos^2(\theta_2-\theta_1)$, the certified rotation rate is largest at $\theta_2-\theta_1=\pi/2$, where
\begin{equation}\label{eq:rotation_cap}
    |\dot\theta_1|\leq\bar\nu=\frac{\sqrt{(l_1^2-\bar\sigma^2)(l_2^2-\bar\sigma^2)}}{|l_1^2-l_2^2|}.
\end{equation}
Hence, for $l_1\neq l_2$, the standard MCBF cannot rotate about joint 1 faster than $\bar\nu$ anywhere in $\Cc$, even though the rotation preserves the safety margin.

\begin{figure}[t]
    \vspace*{4pt}
    \centering
    \includegraphics[width=\columnwidth]{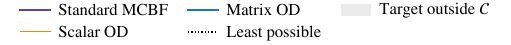}\\[1pt]
    \begin{subfigure}[t]{0.49\columnwidth}
        \centering
        \includegraphics[width=\linewidth]{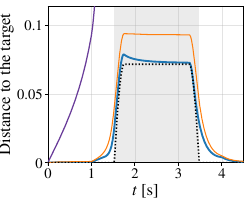}
        \caption{Target crosses~$\bar r$}
        \label{fig:tracking_outer}
    \end{subfigure}\hfill
    \begin{subfigure}[t]{0.49\columnwidth}
        \centering
        \includegraphics[width=\linewidth]{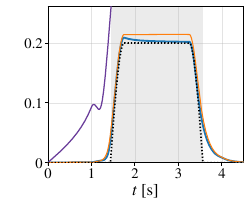}
        \caption{Target crosses~$\ubar r$}
        \label{fig:tracking_inner}
    \end{subfigure}
    \caption{Distance from the end effector to the target for $l_2=0.75$, as in Fig.~\ref{fig:hero}. The distance from the target to the safe set lower bounds the safe tracking error. Matrix OD has the best tracking performance.}
    \label{fig:tracking}
\end{figure}

\subsection{Numerical Experiments}\label{sec:example_numerics}
We fix $l_1=1$, $\bar\sigma=0.6$, and $\kappa=5$\,s$^{-1}$, and consider $l_2\in\{0.75,1,1.25\}$.
For each $l_2$, the target leaves the safe set through either the outer boundary $\bar r$, where the arm is nearly straight, or the inner boundary $\ubar r$, where it is nearly folded, giving six cases.
For $l_2=0.75$, \cref{fig:hero} shows the outer-boundary trajectory, and \cref{fig:tracking} shows the distance to the target for both boundary crossings.

All three filters keep $\lambda_{\min}(\bH)>0$ in every run.
As \cref{fig:tracking} shows, the standard MCBF stays farthest from the target, followed by scalar OD, while matrix OD tracks it most closely.
The standoff of scalar OD follows from \cref{prop:scalarcost}: the smallest scalar decay that certifies the inputs applied by matrix OD has a median of $24$ to $36$, i.e., a penalty distance $\|\tfrac{\omega}{2}\bI-\bOmega_d\|_F$ of $17$ to $25$, so scalar OD settles for a mean $\omega$ of only $6.3$ to $8.3$ and sacrifices tracking instead.
Clarabel averages $0.082$, $0.098$, and $0.137$\,ms per solve for standard MCBF, SOD, and MOD, respectively.

\begin{figure}[t]
    \centering
    \includegraphics[width=\columnwidth]{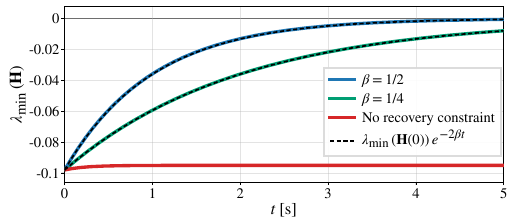}
    \caption{Recovery of matrix OD for $l_2=0.75$ from an unsafe, nearly straight configuration. The arm starts beyond the outer boundary $\bar r$ with the target fixed at its end effector, so the nominal controller~\eqref{eq:arm_nominal} keeps the arm unsafe. With the recovery constraint~\eqref{eq:recovery}, the margin rises at exactly the guaranteed rate $2\beta$ (dashed: $\lambda_{\min}(\bH(0))e^{-2\beta t}$). Without it, the arm stays unsafe.}
    \label{fig:recovery}
\end{figure}

\emph{Recovery.}
To test the recovery constraint~\eqref{eq:recovery}, we initialize the arm with $l_2=0.75$ beyond the outer boundary $\bar r$, at $\lambda_{\min}(\bH)=-0.098$. The target $\by_d$ is fixed at the initial end effector position, so the nominal controller holds the arm in the unsafe region.
\cref{fig:recovery} compares matrix OD with $\beta=\tfrac12$, $\beta=\tfrac14$, and without~\eqref{eq:recovery}.
With the constraint, $\lambda_{\min}(\bH)$ rises at exactly the guaranteed recovery rate $2\beta$ of \cref{sec:recovery}; without the constraint, the arm remains in the unsafe region.

\section{Conclusion}
This paper develops matrix optimal decay, a matrix-multiplier-based optimal decay formulation for matrix control barrier functions.
We showed that scalar decay is sound but not complete at the boundary, and that the proposed matrix condition certifies every control satisfying the matrix tangent cone condition while preserving convexity.
We further established continuity and forward invariance under well-posedness, a bound on deviation relative to the standard filter, and a convex recovery constraint.
Future work will explore the use of matrix decay in other problem domains, including its impact on CBF robustification and potential applications to robotic systems.

\bibliographystyle{IEEEtran}

\bibliography{references}
\end{document}